\documentclass[11pt,oneside, reqno]{amsart}
\usepackage[top=25mm, bottom=25mm, left=25mm, right=25mm]{geometry}
\usepackage{amsfonts,amssymb,amsmath,amsthm,amsxtra}
\usepackage[T1]{fontenc}
\usepackage[utf8]{inputenc}
\usepackage{mathtools}
\usepackage{dsfont}
\usepackage{enumitem}

\numberwithin{equation}{section}
\newtheorem{theorem}[equation]{Theorem}

\newtheorem{lemma}[equation]{Lemma}
\newtheorem{proposition}[equation]{Proposition}

\newtheorem*{ber}{Furstenberg--Bergelson--Leibman conjecture}

\theoremstyle{definition}

\newtheorem{remark}[equation]{Remark}

\newcommand{\ind}[1]{\mathds{1}_{{#1}}}

\begin{document}

\title[The circle method  and pointwise ergodic theorems]{The circle method  and pointwise ergodic theorems}

\author{Mariusz Mirek }
\address[Mariusz Mirek]{
Department of Mathematics,
Rutgers University,
Piscataway, NJ 08854-8019, USA 
\&
Instytut Matematyczny,
Uniwersytet Wroc{\l}awski,
Plac Grunwaldzki 2/4,
50-384 Wroc{\l}aw,
Poland}
\email{mariusz.mirek@rutgers.edu}

\thanks{The author was partially supported  by the NSF CAREER grant (DMS-2236493).}

\begin{abstract} The purpose of this article is to discuss the circle
method and its quantitative role in understanding pointwise almost
everywhere convergence phenomena for polynomial ergodic averaging
operators. Specifically, we will use the circle method to illustrate
that pointwise almost everywhere convergence and norm convergence in
ergodic theory can have fundamentally different natures. More
importantly, these differences may necessitate the use of distinct
types of tools, which can sometimes be more intriguing than the
original problems themselves.
\end{abstract}

\date{\today}

\maketitle

\section{Introduction}

The circle method is a highly successful and powerful technique in
analytic number theory. It was developed by Hardy and Ramanujan
\cite{HR} to obtain the asymptotic formula for the partition function,
which counts the number of unordered representations of a positive
integer $N\in\mathbb Z_+:=\{1, 2, 3,\ldots\}$ as the sum of any number
of positive integers. Subsequently, Hardy and Littlewood \cite{HL1,
HL2} extended these ideas to derive the asymptotic formula in the
context of the Waring problem, which concerns the number of
representations of a positive integer as a sum of $k$-th powers for
$k \geq 2$.

The method was further simplified and strengthen in various ways by
Vinogradov \cite{V}, Kloosterman \cite{Klos} among others.  There is a
vast literature on applications of the circle method to problems in
additive number theory. The books of Iwaniec and Kowalski \cite{IK},
Nathanson \cite{Nat}, Vaughan \cite{Vau}, Vinogradov \cite{Vin}, and
notes of Wooley \cite{Wooleynotes} are excellent references about this
exciting subject.

These days, the circle method has many applications that go far beyond
problems in analytic number theory. Spectacular applications of the
circle method in ergodic theory were discovered by Bourgain. In the
mid of 1980's, Bourgain, in a series of papers \cite{B1, B2, B3},
established various pointwise ergodic theorems with orbits along the
so-called arithmetic sets. The key tool was the circle method. This
line of research was taken up by Stein and his collaborators
\cite{IMSW, MSW, MSW1, SW2}, who developed Bourgain's ideas to study
$L^p$-boundedness of discrete analogues in harmonic analysis.

The purpose of this article is to survey the classical circle method
and its modern perspective that have recently been developed and
demonstrated to be effective in establishing various pointwise ergodic
theorems \cite{BMSWer, Kr, KMT, MSS, MST1, MST2, MSW-survey, MSZ3,
MT}, as well as in understanding discrete operators of Radon type
\cite{IMW, IW, Mes, M10, Pierce0, Slomian, Trojan}. We begin by
examining the context of classical pointwise almost everywhere
convergence problems and norm convergence problems in ergodic
theory. This discussion will elucidate the intricate nature of
pointwise convergence problems and further highlight the necessity of
employing quantitative tools from harmonic analysis, number theory,
and additive combinatorics, all of which are essential for
understanding this mode of convergence, and all of which are interesting in their own right.

\section{Classical pointwise ergodic theorems}
In the early 1930s, von Neumann \cite{vN} established a mean ergodic
theorem, while Birkhoff \cite{BI} proved an almost everywhere
pointwise ergodic theorem (often referred to in the literature as the
individual ergodic theorem). We summarize both ergodic theorems in the
following result.

\begin{theorem}[Birkhoff's and von Neumann's
ergodic theorems]
\label{birk}
Let $(X,\mathcal B(X), \mu)$ be a $\sigma$-finite measure space equipped
with a measure-preserving transformation $T:X\to X$. Then for every
$p\in(1, \infty)$ and every $f\in L^p(X)$ the averages
\begin{align}
\label{eq:63}
A_{N;X, T}^{\mathrm n}f(x):=\mathbb E_{n\in [N]}f(T^{n}x), \qquad x\in X, \qquad N\in\mathbb Z_+
\end{align}
converge almost everywhere on $X$ and in $L^p(X)$ norm as
$N\to\infty$.
\end{theorem}

Here and throughout the paper we will use the following useful
notation $[N]:=(0, N]\cap\mathbb Z$ for any real number $N\ge1$, and
$\mathbb E_{y\in Y}f(y):=\frac{1}{\#Y}\sum_{y\in Y}f(y)$ for any finite set
$Y\neq\emptyset$ and any $f:Y\to\mathbb C$.

\medskip

Theorem \ref{birk}  serves as  an illustrative  example of  the subtle
distinctions     between     norm    convergence     and     pointwise
convergence.  Specifically,  the   conclusion  of  Theorem  \ref{birk}
remains valid when $p = 1$  and $(X, \mathcal{B}(X), \mu)$ is a finite
measure  space.  In fact,  from  the  perspective of  applications  of
ergodic  theorems  --- where  the  statistical  properties of  ergodic
averages in  ergodic measure-preserving systems are  essential --- the
general setting of $\sigma$-finite  measure spaces is not particularly
interesting (since the limits are zero), therefore only finite measure
spaces and  ergodic averages on $L^\infty(X)$  functions are relevant.
Nevertheless,  the $\sigma$-finite  measure spaces  are important  for
pointwise  almost  everywhere  convergence  due  to  the  Calder{\'o}n
transference principle \cite{Cald}, see below.

However, if $(X, \mathcal{B}(X), \mu)$ is a $\sigma$-finite measure
space with $\mu(X) = \infty$, then only pointwise convergence is
guaranteed for $p = 1$ in \eqref{eq:63}, while norm convergence
fails. This can be seen by considering the integer shift system
$(\mathbb Z, \mathcal B(\mathbb Z), \mu_{\mathbb Z})$ endowed with the shift
transformation $S:\mathbb Z\to\mathbb Z$, where $\mathcal B(\mathbb Z)$ is the
$\sigma$-algebra of all subsets of $\mathbb Z$, $\mu_{\mathbb Z}$ is counting
measure on $\mathbb Z$, and $S(x):=x-1$ for  $x\in\mathbb Z$. In this case we
shall abbreviate $A_{N;\mathbb Z, S}^{\mathrm n}f$ to
\begin{align}
\label{eq:59}
A_{N;\mathbb Z}^{\mathrm n}f(x)=\mathbb E_{n\in [N]}f(x-n), \qquad x\in \mathbb Z, \qquad N\in\mathbb Z_+.
\end{align}
Taking $f=\ind{\{0\}}$ in \eqref{eq:59}, we have
$\lim_{N\to \infty}A_{N;\mathbb Z}^{\mathrm n}f(x)=0$, while
$\|A_{N;\mathbb Z}^{\mathrm n}f\|_{\ell^1(\mathbb Z)}=1$ for every $N\in\mathbb Z_+$.
This shows that the norm convergence fails for $p=1$ in the case of a
$\sigma$-finite measure space setting.  \medskip

\subsection{Norm convergence in Theorem \ref{birk}: von Neumann's mean
ergodic theorem} There are many proofs of Theorem \ref{birk} in the
literature; we refer, for instance, to the monographs \cite{EW, Peter}
for more details and historical background. However, here we present
an approach based on Riesz decomposition \cite{Riesz}, which provides
a relatively quick and elementary proof of von Neumann's mean ergodic
theorem and will also be useful in establishing pointwise convergence
in Theorem \ref{birk}.

\medskip

Riesz decomposition \cite{Riesz} asserts that the Hilbert space
$L^2(X)$ can be decomposed into a one-dimensional closed subspace of
$T$-invariant functions and its orthogonal complement. Specifically,
we have ${\rm I}_T\oplus \overline{{\rm J}_T}= L^2(X)$, where
$\overline{{\rm J}_T}$ is the $L^2(X)$ closure of ${\rm J}_T$, and
\begin{align}
\label{eq:61}
 {\rm I}_T:=\{f\in L^2(X): f\circ T=f\},
\qquad \text{ and } \qquad
{\rm J}_T:=\{g-g\circ T: g\in L^2(X)\cap L^{\infty}(X)\}.
\end{align}
We then note that $A_{N;X, T}^{\mathrm n}f=f$ for $f\in {\rm I}_T$,
and $\lim_{N\to\infty}\|A_{N;X, T}^{\mathrm n}h\|_{L^2(X)}=0$ for $h\in {\rm J}_T$,
since
\begin{align}
\label{eq:18}
A_{N;X, T}^{\mathrm n}h=N^{-1}\big(g\circ T-g\circ T^{N+1}\big)
\end{align}
telescopes and converges to $0$ in the $L^2(X)$ norm, whenever $h=g-g\circ T\in {\rm J}_T$. Thus,
$A_{N;X, T}^{\mathrm n}f$ converges in the $L^2(X)$ norm to the
orthogonal projection of $f$ onto ${\rm I}_T$, which establishes 
von Neumann's mean ergodic theorem. 

\medskip

\subsection{Pointwise convergence in Theorem \ref{birk}: Birkhoff's
individual ergodic theorem} We now turn our attention to outline a
classical strategy for handling pointwise convergence problems,  which
is based on a two-step procedure:
\begin{itemize}
\item[(i)] The first step establishes $L^p(X)$ boundedness (when
$p\in(1, \infty)$), or a weak type $(1, 1)$ bound (when $p=1$) of the
corresponding maximal function
$\sup_{N\in\mathbb Z_+}|A_{N;X, T}^{\mathrm n}f(x)|$. More precisely, for  $p\in[1, \infty]$ we are seeking a constant $C_p\in\mathbb R_+$ such that for every $f\in L^p(X)$ the following $L^p(X)$ boundedness holds
\begin{align}
\label{eq:56}
\|\sup_{N\in\mathbb Z_+}|A_{N;X, T}^{\mathrm n}f|\|_{L^p(X)}\le C_p\|f\|_{L^p(X)},
\quad \text{ if } \quad p\in(1, \infty],
\end{align}
or the weak type $(1, 1)$ bound holds
\begin{align}
\label{eq:57}
\sup_{\lambda>0}\lambda\mu\big(\{x\in X: \sup_{N\in\mathbb Z_+}|A_{N;X, T}^{\mathrm n}(x)|>\lambda\}\big)\le C_1\|f\|_{L^1(X)},
\quad \text{ if } \quad p=1.
\end{align}
Inequalities \eqref{eq:56} and \eqref{eq:57} for the integer shift system with $\sup_{N\in\mathbb Z_+}|A_{N;\mathbb Z}^{\mathrm n}f|$
in place of $\sup_{N\in\mathbb Z_+}|A_{N;X, T}^{\mathrm n}f|$ follow from the Hardy--Littlewood
maximal theorem, see for instance \cite{bigs}. Then by the Calder{\'o}n transference principle \cite{Cald}, they can be
transferred from the integer shift system to arbitrary measure-preserving system giving \eqref{eq:56} and \eqref{eq:57}.
 Having  these maximal estimate in hands one can 
easily prove that the set
\begin{align*}
\qquad \qquad \mathbf P\mathbf C[L^{p}(X)]=\{f\in L^{p}(X): \lim_{N\to\infty}A_{N;X, T}^{\mathrm n}f \text{
exists $\mu$-almost everywhere on $X$}\}
\end{align*}
is closed in $L^{p}(X)$ for any $p\in[1, \infty)$.

\item[(ii)] In the second step one shows that
$\mathbf P\mathbf C[L^{p}(X)]= L^{p}(X)$. In view of the first step
one has to find a dense class of functions in
$L^{p}(X)$ for which we have pointwise convergence. Assuming that $p=2$,
and invoking Riesz decomposition \cite{Riesz}, a good
candidate is the space ${\rm I}_T\oplus {\rm J}_T\subseteq L^2(X)$ with ${\rm I}_T, {\rm J}_T$ as in \eqref{eq:61}.
We then note that $A_{N;X, T}^{\mathrm n}f=f$ for  $f\in {\rm I}_T$, and $\lim_{N\to\infty}A_{N;X, T}^{\mathrm n}h=0$ almost everywhere on $X$ for $h\in {\rm J}_T$, by \eqref{eq:18}. This establishes pointwise almost
everywhere convergence of $A_{N;X, T}^{\mathrm n}$ on ${\rm I}_T\oplus {\rm J}_T$, which is dense
in $L^2(X)$. These two steps guarantee that $\mathbf P\mathbf C[L^{2}(X)]= L^{2}(X)$. Consequently,
$A_{N;X, T}^{\mathrm n}$ converges pointwise on $L^{p}(X)\cap L^{2}(X)$ for any $p\in[1, \infty)$. Since 
$L^{p}(X)\cap L^{2}(X)$ is dense in $L^{p}(X)$  we also conclude, in view of the first step, that
$\mathbf P\mathbf C[L^{p}(X)]= L^{p}(X)$, and a brief outline of the proof of Theorem \ref{birk} is completed.
\end{itemize}

This two-step procedure originates from the Banach continuity
principle \cite{Banach} and illustrates a general rule for addressing
pointwise convergence problems. It is particularly effective for
questions of pointwise convergence that arise in harmonic analysis, as
there are many natural dense subspaces within Euclidean spaces (like
Schwartz or even smooth bump functions) that can be utilized to verify
step 2. However, in ergodic theory, the scenario can differ
significantly when working with abstract measure spaces, as
demonstrated by Bourgain \cite{B1, B2, B3}. We will explore more
examples below.

There is a vast literature \cite[Section X, p. 454]{bigs} and
\cite[Section 1]{Garsia}, extending the Banach continuity principle
\cite{Banach}, and demonstrating that pointwise convergence for many
operators is essentially equivalent to the boundedness of their
corresponding maximal functions. Therefore, the maximal inequalities
from \eqref{eq:56} and \eqref{eq:57} can be regarded as a quantitative
form of pointwise convergence and highlight the distinction between
norm and pointwise convergence.

In the first step of the two-step procedure, a very important tool is
the Calder{\'o}n transference principle \cite{Cald}, which reduces the
maximal inequalities \eqref{eq:56} and \eqref{eq:57} from abstract
measure-preserving systems $(X, \mathcal{B}(X), \mu)$ to the integer
shift system. We also refer to \cite{Kosz}. This significant reduction allows us to utilize the
algebraic structure of $\mathbb Z$, where ergodic averages become
convolution operators. For instance, the average
$A_{N;\mathbb Z}^{\mathrm n}f(x)=\mathbb E_{n\in [N]}f(x-n)$ from \eqref{eq:59}
represents a convolution of $f$ with the kernel
$K_N=N^{-1}\sum_{n\in[N]}\ind{\{n\}}$.  This interpretation, in turn,
enables us to employ Fourier methods or almost orthogonality
techniques on $\mathbb{Z}$, which are generally not available in
abstract measure-preserving systems. Therefore, the necessity of
working with the set of integers equipped with a counting measure,
which is a $\sigma$-finite measure space, is the primary reason
(though a very important one) why we formulate our theorems in this
paper for $\sigma$-finite measure spaces.

We also have to emphasize that the Calder{\'o}n transference principle
\cite{Cald, Kosz} only transfers quantitative bounds (like maximal
inequalities \eqref{eq:56} and \eqref{eq:57}) that imply pointwise
almost everywhere convergence, but does not transfer pointwise almost everywhere convergence
itself. In fact, in the integer shift system, pointwise convergence is
implied by norm convergence, since the $\ell^\infty(\mathbb Z)$ norm is
dominated by the $\ell^2(\mathbb Z)$ norm. Hence, from now on we will only be
concerned with discussing and proving quantitative bounds.

Finally, we would like to point out another subtle difference between
norm convergence and pointwise almost everywhere convergence, using
Theorem \ref{birk}. Namely, von Neumann in \cite{vN} proved mean
ergodic theorem for the so-called uniform averages, which are given by \eqref{eq:63} with
shifted averages
$\mathbb E_{n\in[N]\setminus[M]}:=\frac{1}{N-M}\sum_{n=M+1}^N$ in
place of the ordinary averages
$\mathbb E_{n\in[N]}:=\frac{1}{N}\sum_{n=1}^N$. These two averages
coincide for $M=0$.  The asymptotic behavior of uniform averages in
the $L^2(X)$ norm is pivotal in combinatorial applications; we refer
to~\cite{Ber1, Ber2} for motivations and an
extensive literature.

By the dominated convergence theorem, the pointwise almost everywhere
convergence of the uniform averages
$\mathbb E_{n\in[N]\setminus[M]}f(T^{n}x )$ as $N - M \to \infty$
implies their norm convergence for all $f\in L^{\infty}(X)$ on a
probability space $(X, \mathcal B(X), \mu)$. Interestingly, pointwise
almost everywhere convergence for uniform averages $\mathbb E_{n\in[N]\setminus[M]}f(T^{n}x )$ can fail, as shown by del Junco and Rosenblatt~\cite{dJR}. This stands in sharp contrast to the behavior of uniform averages in norm, highlighting another subtle difference between norm convergence and pointwise almost everywhere convergence. Therefore, it only makes sense to study pointwise almost everywhere convergence for ordinary averages $\mathbb E_{n\in[N]}:=\frac{1}{N}\sum_{n=1}^N$.

\begin{remark}
\label{rem:9}
This last phenomenon  reminds a bit a situation from Fourier series or the theory of complete orthonormal systems  that any
complete orthonormal system $(\Phi_n)_{n\in\mathbb Z_+}$ on $[0, 1]$ admits a rearrangement
$(\Phi_{\sigma(n)})_{n\in\mathbb Z_+}$ which is such that for some function $f\in L^2([0, 1])$ the partial sums
$\sum_{n\in[N]}\langle f, \Phi_{\sigma(n)}\rangle \Phi_{\sigma(n)}$ diverge almost everywhere, while the norm convergence remains unchanged
\[
\lim_{N\to\infty}\Big\|\sum_{n\in[N]}\langle f, \Phi_{\sigma(n)}\rangle \Phi_{\sigma(n)}-f\Big\|_{L^2([0, 1])}=
\lim_{N\to\infty}\Big\|\sum_{n\in[N]}\langle f, \Phi_{n}\rangle \Phi_{n}-f\Big\|_{L^2([0, 1])}=0.
\]
In particular, we can take the canonical orthonormal system
$(e^{2\pi i \xi n})_{n\in\mathbb Z}$ on $[0, 1]$, for which the
partial sums are almost everywhere convergent for all
$f\in L^2([0, 1])$, see Carleson's paper \cite{Carleson} and spoil it by a
suitable permutation.  We refer to \cite{Ul} for more details on this
subject, as well as to \cite{Garsia}.
\end{remark}

\section{Bourgain's pointwise ergodic theorem}
In 1933, Khintchin \cite{Khin}, using Birkhoff's ergodic theorem
\cite{BI}, generalized the classical equidistribution result of  Weyl \cite{Weyl-10} 
to the following ergodic equidistribution result: for any irrational
$\theta \in {\mathbb R}$, for any Lebesgue measurable set
$E\subseteq [0,1)$, and for almost every $x\in {\mathbb R}$, one has
\begin{align*}
\lim_{N\to\infty} \frac{\# \{ n\in[N]: \{x + n \theta\} \in E \}}{N}  =  |E|,
\end{align*}
where $\{x\}$ denotes the fractional part of $x\in\mathbb R$.

In 1916, Weyl
\cite{Weyl} extended the classical equidistribution theorem to general
polynomial sequences $(\{P(n)\})_{n\in\mathbb Z}$ having at least one
irrational coefficient. Then it was natural to ask whether a
pointwise ergodic extension of Weyl's equidistribution theorem
holds. This question was posed by Bellow \cite{Bel} and Furstenberg
\cite{F} in the early 1980's.  Specifically, they asked whether, for any
polynomial $P \in {\mathbb Z}[{\rm n}]$ with integer coefficients and
for any invertible measure-preserving transformation $T: X \to X$ on a
probability space $(X,\mathcal B(X), \mu)$, the averages
\begin{align}
\label{eq:66}
A_{N;X, T}^{P(\mathrm n)}f(x)=\mathbb E_{n\in [M]}f(T^{P(n)}x), \qquad x\in X, \qquad N\in\mathbb Z_+
\end{align}
converge almost everywhere on $X$ as $N\to\infty$, for any
$f\in L^{\infty}(X)$.

An affirmative answer to this question was given by Bourgain in series
of groundbreaking papers \cite{B1,B2,B3} which we summarize in
the following theorem.
\begin{theorem}[Bourgain's ergodic theorem]
\label{bourgain}
Let $(X,\mathcal B(X), \mu)$ be a $\sigma$-finite measure space
equipped with an invertible measure-preserving transformation
$T:X\to X$, and $P\in \mathbb Z[{\rm m}]$ be a polynomial with integer
coefficients. Then for every $p\in(1, \infty)$ and every $f\in L^p(X)$
the averages $A_{N;X, T}^{P}f$ from \eqref{eq:66} converge almost
everywhere on $X$ and in $L^p(X)$ norm as $N\to\infty$.
\end{theorem}

In contrast to Birkhoff's theorem, if $P\in\mathbb Z[\rm m]$ is a polynomial
of degree at least two, the pointwise convergence in Theorem
\ref{bourgain} at the endpoint for $p=1$ may fail. It was shown by
Buczolich and Mauldin \cite{BM} for $P(m)=m^2$ and by LaVictoire
\cite{LaV1} for $P(m)=m^k$ for any $k\ge2$. This sharply contrasts
with what occurs in the continuous analogues of ergodic averages \cite{bigs} and
illustrates that any intuition we develop in Euclidean harmonic
analysis --- where sums are replaced by integrals --- can fail
dramatically in discrete problems, making this subject particularly
interesting.

\subsection{Norm convergence in Theorem \ref{bourgain}: Furstenberg's
mean ergodic theorem} The norm convergence in Theorem \ref{bourgain}
was established by Furtsneberg \cite{Fur2}. Assuming that $\mu(X)=1$
and $T:X\to X$ is totally ergodic (i.e. $T^n$ is ergodic for every
$n\in\mathbb Z_+$) one can prove that
\begin{align}
\label{eq:22}
\lim_{N\to\infty}\Big\|A_{N;X, T}^{P(\mathrm n)}f-\int_Xfd\mu\Big\|_{L^2(X)}=0.
\end{align}
We now briefly outline \eqref{eq:22} relying on the van der Corput inequality in Hilbert spaces \cite{EW}.

\begin{proposition}[van der Corput's inequality in Hilbert spaces, see \cite{EW}]
Let $(u_n)_{n \in \mathbb Z_+}$ be a bounded sequence in a Hilbert space
$(\mathbb H, \|\cdot\|)$. Define a sequence $(s_n)_{n \in \mathbb Z_+}$ of real numbers by
\[
s_h = \limsup_{N \rightarrow \infty} | \mathbb E_{n\in[N]}\langle u_{n+h}, u_n \rangle |.
\]
If $\lim_{H \rightarrow \infty} \mathbb E_{h\in[H]}s_h = 0$,
then $\lim_{N \rightarrow \infty} \|\mathbb E_{n\in[N]} u_n \| = 0$. 
\end{proposition}

To prove \eqref{eq:22} we proceed by induction on the degree of
$P$. If $P$ is linear, then the result follows from von Neumann's mean
ergodic theorem and the total ergodicity of $T$. Assume that
${\rm{deg}}(P) \geq 2$ and by linearity of $A_{N;X, T}^{P(\mathrm n)}$
we can assume that $\int_X fd\mu = 0$. We will use the van der Corput
inequality for $\mathbb H=L^2(X)$ to prove that
$\lim_{N \rightarrow \infty} \mathbb E_{n\in[N]} u_n = 0$ in $L^2(X)$ with
$u_n = f \circ T^{P(n)}$.  Fix $h \in \mathbb Z_+$ and note that
$n\mapsto P(n+h) - P(n)$ is a polynomial of degree at most
${\rm{deg}}(P)-1$. Therefore, by the induction hypothesis, we have 
$\lim_{N \rightarrow \infty} \mathbb E_{n\in[N]} f \circ T^{P(n+h) - P(n)} = 0$
in $L^2(X)$ for
every $h \in \mathbb Z_+$. By the Cauchy--Schwarz inequality
\[
\lim_{N\rightarrow \infty} \mathbb E_{n\in[N]} \langle u_{n+h}, u_n \rangle =
\lim_{N\rightarrow \infty} \big\langle \mathbb E_{n\in[N]} f \circ T^{P(n+h) - P(n)}, f \big\rangle = 0.
\]
This establishes the hypothesis of the van der Corput lemma, which
consequently implies \eqref{eq:22}.

Furstenberg, using \eqref{eq:22}, proved the polynomial Poincar{\'e}
recurrence theorem \cite{Fur2} and obtained an ergodic proof of S{\'a}rk{\"o}zy
theorem \cite{Sark1, Sark2}. The proof of \eqref{eq:22} relies solely
on Hilbert space methods, and it is not surprising that the
corresponding pointwise almost everywhere convergence will necessitate
more delicate quantitative information to follow the two-step
procedure.

\medskip

\subsection{Pointwise  convergence in Theorem \ref{bourgain}: Bourgain's
ergodic theorem} Bourgain in \cite{B1,B2,B3}  used the two-step procedure  to prove
Theorem \ref{bourgain}. In the first step, Bourgain showed that for all
$p\in(1, \infty]$, there is $C_{p, P}>0$ such that for every
$f\in L^p(X)$ one has
\begin{align}
\label{eq:33}
\big\|\sup_{N\in\mathbb Z_+}|A_{N; X, T}^{P}f|\big\|_{L^p(X)}
\le C_{p, P}\|f\|_{L^p(X)}.
\end{align}

However, in the second step, identifying a dense class of functions
for which pointwise almost everywhere convergence in Theorem
\ref{bourgain} holds turned out to be a fairly challenging problem.
Although Riesz decomposition ${\rm I}_T\oplus {\rm J}_T$ (as for
$A_{N;X, T}^{{\mathrm n}}$) still makes sense, it is not sufficient if
$\deg P\ge2$.  Even for the squares $P(n)=n^2$ it is not clear whether
$\lim_{N\to\infty}A_{N;X, T}^{{\mathrm n}^2}h=0$ for $h\in {\rm J}_T$,
since $(n+1)^2-n^2=2n+1$ have unbounded gaps, and 
$A_{N;X, T}^{{\mathrm n}^2}h$ does not telescope for $h\in {\rm J}_T$
anymore.

\medskip

To overcome this problem, Bourgain proposed controlling  the so-called oscillation
semi-norms. Let $\mathbb J\subseteq \mathbb N=\{0, 1, 2,\ldots\}$ be so that $\#\mathbb J\ge2$,
let $I=(I_j:j\in\mathbb N\cap [0, J])$ be a strictly increasing sequence of length $J+1$ for some
$J\in\mathbb Z_+$, which takes values in $\mathbb J$, and recall that for any
sequence $(\mathfrak a_{t}: t\in\mathbb J)\subseteq\mathbb C$,  and any exponent
$1 \leq r < \infty$, the $r$-oscillation
seminorm is defined by
\begin{align}
\label{eq:38}
O_{I, J}^r(\mathfrak a_{t}: t \in \mathbb J):=
\Big(\sum_{j=0}^{J-1}\sup_{\substack{I_j\le t<I_{j+1}\\ t\in \mathbb J}}
|\mathfrak a_{t} - \mathfrak a_{I_j}|^r\Big)^{1/r}.
\end{align}

Using \eqref{eq:38} with $r=2$,  Bourgain proved that
for any $\tau>1$, any 
$I=(I_j:{j\in\mathbb N})\subseteq \mathbb D_{\tau}:=\{\lfloor \tau^n\rfloor: n\in\mathbb N\}$
such that $I_{j+1}>2I_j$ for all $j\in\mathbb N$, and any $f\in L^2(X)$ one
has
\begin{align}
\label{eq:34}
\big\|O_{I, J}^2(A_{N; X, T}^{P}f:N\in\mathbb D_\tau)\big\|_{L^2(X)}
\le C_{I, \tau}(J)\|f\|_{L^2(X)}, \qquad J\in\mathbb Z_+,
\end{align}
where $C_{I, \tau}(J)$ is a constant depending on $I$ and $\tau$ and
satisfies
\begin{align}
\label{eq:41}
\lim_{J\to\infty} J^{-1/2}C_{I, \tau}(J)=0.
\end{align}

Bourgain \cite{B1, B2, B3} had the ingenious insight that inequality
\eqref{eq:34} in conjunction with \eqref{eq:41} is sufficient to
establish pointwise convergence of $A_{N; X, T}^{P}f$ for any
$f\in L^2(X)$.  On the one hand, \eqref{eq:34} is quite close to the
maximal inequality; by applying \eqref{eq:33} with $p=2$ we can derive
\eqref{eq:34} with a constant at most $J^{1/2}$. Moreover, in view of
Remark \ref{rem:9}, the supremum in the definition of oscillations
given in \eqref{eq:38} cannot be omitted.  On the other hand, any
improvement (better than $J^{1/2}$) for the constant in \eqref{eq:34}
implies \eqref{eq:41} and, consequently, guarantees pointwise
almost everywhere convergence of $A_{N; X, T}^{P}f$ for any $f\in L^2(X)$. For further
details, including the multi-parameter setting, see
\cite{MSW-survey}. Thus, from this perspective, the combination of
inequalities \eqref{eq:34} and \eqref{eq:41} represents the minimal
quantitative requirement necessary to establish pointwise almost
everywhere convergence.

Bourgain's papers \cite{B1, B2, B3} represented a significant
breakthrough in ergodic theory, providing a variety of new tools ---
ranging from harmonic analysis and number theory to probability and the
theory of Banach spaces --- to study pointwise convergence problems in
a broad sense. In \cite{B3}, a complete proof of Theorem
\ref{bourgain} is presented, and the concepts of $r$-variations and
$\lambda$-jumps are introduced as two additional important
quantitative tools for investigating pointwise almost everywhere
convergence problems. The $r$-variation seminorms and $\lambda$-jumps  are defined as follows. 

\begin{enumerate}[label*={\arabic*}.]
\item For any $\mathbb I\subseteq \mathbb N$, any sequence
$(\mathfrak a_t: t\in\mathbb I)\subseteq \mathbb C$, and any exponent
$1 \leq r < \infty$, the $r$-variation semi-norm is defined to be
\begin{align*}
V^{r}(\mathfrak a_t: t\in\mathbb I):=
\sup_{J\in\mathbb Z_+} \sup_{\substack{t_{0}<\dotsb<t_{J}\\ t_{j}\in\mathbb I}}
\Big(\sum_{j=0}^{J-1}  |\mathfrak a_{t_{j+1}}-\mathfrak a_{t_{j}}|^{r} \Big)^{1/r},
\end{align*}
where the latter supremum is taken over all finite increasing sequences in
$\mathbb I$. 

\item For any $\mathbb I\subseteq \mathbb N$ and any $\lambda>0$, the $\lambda$-jump
counting function of a sequence  $(\mathfrak a_t: t\in\mathbb I)\subseteq \mathbb C$ is defined by
\begin{align*}
N_\lambda(\mathfrak a_t:t\in\mathbb I):=\sup \{ J\in\mathbb N : \exists_{\substack{t_{0}<\ldots<t_{J}\\ t_{j}\in\mathbb I}}
: \min_{0 \le j \le J-1} |\mathfrak a_{t_{j+1}}-\mathfrak a_{t_{j}}| \ge \lambda \}.
\end{align*}
\end{enumerate}

If $V^r(\mathfrak a_n(x): n\in \mathbb Z_+)$ is finite $\mu$-almost everywhere
on $X$, then the underlying sequence is Cauchy and the limit
$\lim_{n\to\infty}\mathfrak a_n(x)$ exists $\mu$-almost everywhere on
$X$. This is a very useful feature of $r$-variations since it allows us
to establish pointwise almost everywhere convergence  in one step.
However, $r$-variational estimates are harder to obtain than maximal ones,
and this is the price that we have to pay. Namely, one has
\begin{align*}
\sup_{n\in \mathbb Z_+}|\mathfrak a_n(x)|\le V^r(\mathfrak a_n(x): n\in \mathbb Z_+)+|\mathfrak a_{n_0}(x)| \quad \text{ for any } \quad 1\le r<\infty, \quad  n_0\in \mathbb Z_+.
\end{align*}

The $r$-variations for a family of bounded martingales
$(\mathfrak f_n:X\to \mathbb C:n\in \mathbb Z_+)$ were studied by L{\'e}pingle
\cite{Lep}, who showed that for all $r\in(2, \infty)$ and
$p\in(1, \infty)$, there is  $C_{p, r}>0$ such that the following inequality,
\begin{align}
\label{eq:15}
\|V^r(\mathfrak f_n: n\in \mathbb Z_+)\|_{L^p(X)}
\le C_{p, r}\sup_{n\in \mathbb Z_+}\|\mathfrak f_n\|_{L^p(X)},
\end{align}
holds with a sharp ranges of exponents, see \cite{JG} for a
counterexample at $r=2$.  In \cite{Lep}, weak type $(1,1)$ variants of
\eqref{eq:15} were proved as well.  We also refer to \cite{PX,B3} and
\cite{MSZ1} for generalizations and different proofs of
\eqref{eq:15}. Inequality \eqref{eq:15} is an extension of Doob's
maximal inequality for martingales and gives a quantitative form of
the martingale convergence theorem.  Bourgain rediscovered inequality
\eqref{eq:15} in his seminal paper \cite{B3}, where it was used to
address the issue of pointwise almost everywhere convergence of ergodic
averages along polynomial orbits.  This initiated a systematic study
of quantitative estimates in harmonic analysis and ergodic theory
which resulted in a vast literature: in ergodic theory \cite{BM, jkrw,
Kr, KMT, LaV1, Mes, MSS, MST2, MSZ3, Trojan}, in discrete harmonic
analysis \cite{IMMS, IMSW, IMW, IW, MSW, MSW1, M10, MSZ1, MT, Pierce0,
Pierce, RW}, and in classical harmonic analysis \cite{ JG, JSW, KLMP,
MST1, OSTTW, MSZ2}.  The latter shows that, in our applications, only
$r > 2$ and $p > 1$ matter, and, in fact, this is the best we can
expect due to the L{\'e}pingle inequality.

\section{A first glimpse at the circle method}
In the last two sections we have examined the differences in
strategies for establishing convergence in norm and pointwise almost
everywhere in the context of classical and Bourgain's ergodic
averages. The key conclusion is that pointwise almost everywhere
convergence necessitates tools that have a more quantitative nature,
whereas norm convergence is primarily qualitative and can be addressed
fairly directly, at least in the linear case.  We will now illustrate
the types of quantitative tools that can be employed to control
maximal functions or $r$-variational semi-norms.
The primary tool for establishing pointwise convergence of polynomial
ergodic averages \eqref{eq:66} is the circle method. These ideas originate
from Bourgain's papers \cite{B1, B2, B3}. Since Bourgain's work, this
method has undergone considerable changes and developments \cite{M10, MSS,  MST1, MST2, MSZ3, MT}, which can
be summarized as follows:

\begin{enumerate}[label*={(\alph*)}]
\item\label{I1} To control the minor arc contribution, we apply
Plancherel's theorem and Weyl's inequality for exponential sums.

\smallskip

\item\label{I2} To control the major arc contribution, we use multifrequency harmonic
  analysis in the spirit of the Ionescu--Wainger multiplier theorem (see
  Theorems 3.3 and 3.27 in \cite{KMPW}).
\end{enumerate}

In the next two sections, we illustrate the ideas from \ref{I1} and \ref{I2} in the context
of $r$-variational estimates for \eqref{eq:66}. Let $P\in\mathbb Z[\rm n]$ be
a polynomial of degree $d\ge2$ and
$\mathbb D_{\tau}:=\{\lfloor \tau^n\rfloor: n\in\mathbb N\}$ be a lacunary
sequence for a fixed $\tau>1$. Specifically, we prove that for every
$p\in(1, \infty)$ and $r\in(2, \infty]$ there exists a constant
$C_{p, r, \tau; P}\in\mathbb R_+$ such that for every $f\in L^p(X)$ we
have
\begin{align}
\label{eq:48}
\|V^r(A_{N;X, T}^{P}f: N\in\mathbb D_{\tau})\|_{L^p(X)}\le C_{p, r, \tau; P}^P\|f\|_{L^p(X)}. 
\end{align}
Once inequality \eqref{eq:48} is established for any $\tau>1$, it
implies Theorem \ref{bourgain}, see \cite[Lemma 1.5]{RW}.

In fact, with more effort, inequality \eqref{eq:48} can also be
established with $\mathbb Z_+$ in place of $\mathbb D_{\tau}$ with the implied
constant remaining independent of $\tau>1$, see \cite{MST2, MSZ3}.  By
the Calder{\'o}n transference principle \cite{Cald}, which transfers
the problem to the integer shift system
$(\mathbb Z, \mathcal B(\mathbb Z), \mu_{\mathbb Z})$ with the integer shift $S:\mathbb Z\to\mathbb Z$
given by $S(x)=x-1$ for $x\in\mathbb Z$, it suffices to show that for every
$f\in \ell^p(\mathbb Z)$ we have
\begin{align}
\label{eq:17}
\|V^r(A_{N;\mathbb Z, S}^{P}f: N\in\mathbb D_{\tau})\|_{\ell^p(\mathbb Z)}\le C_{p, r, \tau; P}^P\|f\|_{\ell^p(\mathbb Z)}. 
\end{align}
From now on, we shall abbreviate $A_{N; \mathbb Z, S}^{P}$ to
\begin{align}
\label{eq:24}
A_{N;\mathbb Z}^{P}f(x)=\mathbb E_{n\in [N]}f(x-P(n)), \quad x\in\mathbb Z, \quad N\in\mathbb Z_+.
\end{align}
The average from \eqref{eq:24} is a convolution operator of the kernel
$K_N^P=N^{-1}\sum_{n\in[N]}\ind{\{P(n)\}}$ with a function $f$. This
is a very important outcome of this reduction. The averages from
\eqref{eq:24} are discrete averaging operators of Radon type, and
their integral counterparts are well understood through classical
harmonic analysis methods, such as the Littlewood–Paley theory (see
\cite{bigs} and also \cite{RdF}).

\subsection{Classical circle method} This reduction is essential for
employing Fourier tools on $\mathbb Z$.  Recall that
$e(\xi):=e^{2\pi i \xi}$ for $\xi\in\mathbb R$, and the Fourier
transform on $\mathbb Z$ and $\mathbb T:=\mathbb R/\mathbb Z$ is defined respectively by
\begin{align*}
\mathcal F_{\mathbb Z}f(\xi)&:=\sum_{n\in\mathbb Z}e(\xi n)f(n) \quad \text{ for } \quad \xi\in\mathbb T, \ f\in\ell^1(\mathbb Z),\\
\mathcal F_{\mathbb Z}^{-1}f(n)&:=\int_{\mathbb T}e(-\xi n)f(\xi)d\xi \quad \text{ for } \quad n\in\mathbb Z, \ f\in L^1(\mathbb T).
\end{align*}
Using the Fourier transform, we have $\mathcal F_{\mathbb Z}A_{N; \mathbb Z}^{P}(f)(\xi)=m_{N}(\xi)\mathcal F_{\mathbb Z}f(\xi)$ for $\xi\in\mathbb T$,
where the multiplier $m_N$ is the exponential sum of the form 
\begin{align}
\label{eq:28}
m_{N}(\xi) := \mathcal F_{\mathbb Z}K_N^P(\xi)= \mathbb E_{n\in [N]}e(\xi P(n)).
\end{align}
The classical circle method can be used to understand the nature of
the multiplier $m_{N}$ by analyzing the major and minor arcs. This will require the concepts of canonical
fractions and their corresponding major arcs. For $N_1, N_2\in \mathbb R_+$
we define the set of \emph{canonical fractions} by
\begin{align}
\label{eq:4}
\mathcal R_{\le N_1} := \left\{a/q\in\mathbb T\cap\mathbb Q: q\in[N_1] \text{ and } (a, q)=1\right\},
\end{align}
and the corresponding set of \emph{major arcs} by setting
\begin{align}
\label{eq:23}
\mathfrak M_{\le N_2}(\mathcal R_{\le N_1}) := \bigcup_{\theta\in \mathcal R_{\le N_1}}[\theta-N_2, \theta+N_2].
\end{align}
In what follows, the major arcs will be always considered with $N_1:=\delta^{-C_1}$ and $N_2:=N^{-d}\delta^{C_2}$ for $N\ge 1$ and $\delta\in(0, 1]$ and some fixed parameters $C_1, C_2\in\mathbb R_+$. If $N\ge1 $ is sufficiently large in terms of $\delta$, say
$N>3\delta^{-3C}$, where $C=\max\{C_1, C_2\}$, then the intervals that comprise the set of
major arcs are narrow and disjoint in $\mathbb T$.
The  set of \emph{minor arcs} is then defined as the complement of the
set of major arcs in $\mathbb T$.
We formulate Weyl's estimate for the multiplier $m_N(\xi)$ as follows:
\begin{proposition}[Weyl's inequality for the multiplier $m_N$]
For every $C\in\mathbb R_+$ there exists a small constant
$c\in(0, 1)$ such that for all $N\ge 1$ and $\delta\in (0, 1]$,
whenever $\xi$ lies outside of the major arc
$\mathfrak M_{\le N^{-d}\delta^{-C}}(\mathcal R_{\le \delta^{-C}})$,
we have
\begin{align}
\label{eq:288}
|m_{N}(\xi)|\le c^{-1}(\delta^c+N^{-c}).
\end{align}
\end{proposition}
In fact, inequality \eqref{eq:288} is the classical Weyl sum estimate
for normalized exponential sums; see for instance
\cite[Exercise~1.1.21, p.~16]{Tho}. The Weyl-type estimate
\eqref{eq:288}, combined with the spectral theorem, can be
used to reprove Furstenberg's mean ergodic theorem as stated in
 \eqref{eq:22}.

Here and throughout, we write $A \lesssim B$ for two quantities
$A, B\ge0$ if there is an absolute constant $C>0$, (which may change
from line to line), such that $A\le CB$. We will write $A \simeq B$
when $A \lesssim B\lesssim A$.  We will also write $\lesssim_{\delta}$
or $\simeq_{\delta}$ to emphasize that the implicit constant depends
on $\delta$.

Taking $\delta=N^{-\varepsilon}$ with $\varepsilon\in(0, 1)$ sufficiently small, say
$\varepsilon C<1/4$, and $N$ sufficiently large, we gain a negative power of $N$ in
\eqref{eq:288} and this, combined with Plancherel's theorem, yields
\begin{align}
\label{eq:55}
\|A_{N; \mathbb Z}^{P}f\|_{\ell^2(\mathbb Z)}\lesssim_{\varepsilon} N^{-c\varepsilon}\|f\|_{\ell^2(\mathbb Z)},
\end{align}
for every function $f\in \ell^2(\mathbb Z)$ whose Fourier transform
vanishes on major arcs
$\mathfrak M_{\le N^{-d}\delta^{-C}}(\mathcal R_{\le \delta^{-C}})$. Estimate
\eqref{eq:55} is critical to control $A_{N; \mathbb Z}^{P}$ on minor arcs as
$N^{-c\varepsilon}$ is summable for all $N\in\mathbb D_{\tau}$.

\subsection{Basic Ionescu--Wainger multiplier theorem}
One of the challenges is how to interpret inequality \eqref{eq:55} in
$\ell^p(\mathbb Z)$ spaces with $p\in(1, \infty)$.  For example, it is not
clear how to make use of Weyl's inequality, as Plancherel's theorem is
not available in $\ell^p(\mathbb Z)$ when $p \neq 2$.

A way to overcome this difficulty is to proceed as follows:
\begin{enumerate}[label*={(\roman*)}]
\item Localize the  major arcs in the $\ell^p(\mathbb Z)$ spaces for  $p\in(1, \infty)$ by using  smooth bump functions.
 
\smallskip
\item Instead of working with the exponential sum $m_N$ itself, one works with
the corresponding averaging operator $A_{N; \mathbb Z}^{P}$. We recall that $\mathcal F_{\mathbb Z}A_{N; \mathbb Z}^{P}f=m_{N}\mathcal F_{\mathbb Z}f$ for all $f\in\ell^2(\mathbb Z)$.

\end{enumerate}

Taking a smooth even cutoff function $\eta \colon \mathbb R\to[0, 1]$ such
that
\begin{align}
\label{eq:58}
\ind{[-1/4, 1/4]}\le \eta\le \ind{(-1/2, 1/2)},
\end{align}
we define for
$N_1, N_2\in\mathbb R_+$ the smooth projection operator
$\Pi[\le N_1, \le N_2] \colon \ell^2(\mathbb Z)\to \ell^2(\mathbb Z)$ by setting
\begin{align}
\label{eq:290}
\mathcal F_{\mathbb Z}\left(\Pi\left[\le N_1, \le N_2\right]f\right)(\xi)
 := \sum_{\theta\in \mathcal R_{\le N_1}}\eta\left(N_2^{-1}(\xi-\theta)\right)\mathcal F_{\mathbb Z}f(\xi).
\end{align}
These projections will be called the Ionescu--Wainger projections and will allow us to
effectively localize the major arcs. If $N\ge1 $ is sufficiently large in terms of $\delta$, say
$N>3\delta^{-3C}$, we note that
\begin{itemize}
\item[(a)] The operator
$\Pi[\le \delta^{-C}, \le N^{-d}\delta^{-C}]$ is a contraction on
$\ell^2(\mathbb Z)$, by
Plancherel's theorem,  since the arcs in
$\mathfrak M_{\le N^{-d}\delta^{-C}}(\mathcal R_{\le \delta^{-C}})$ are disjoint, and $\sum_{\theta\in \mathcal R_{\le N_1}}\eta\left(N_2^{-1}(\xi-\theta)\right)\le 1$.

\item[(b)] The operator $\Pi[\le \delta^{-C}, \le N^{-d}\delta^{-C}]$
is bounded on $\ell^p(\mathbb Z)$ for all $p\in[1, \infty]$, since the estimate
$\# \mathcal R_{\le \delta^{-C}}\le \delta^{-2C}$ implies the crude
bound
\begin{align}
\label{eq:292}
\big\|\Pi[\le \delta^{-C}, \le N^{-d}\delta^{-C}]f\big\|_{\ell^p(\mathbb Z)}
\lesssim \delta^{-2C} \|f\|_{\ell^p(\mathbb Z)}.
\end{align}
\end{itemize}

The bound from \eqref{eq:292} will not be very useful. However, the
$\ell^p(\mathbb Z)$ norms of the Ionescu--Wainger projections \eqref{eq:290}
have reasonably good growth in terms of the size of the set of
canonical fractions $\mathcal R_{\le \delta^{-C}}$, which can be subsummed as the following theorem:
\begin{theorem}[Ionescu--Wainger theorem for projections]
\label{thm:IWproj}
For every $p\in(1, \infty)$ and any parameters $C_1, C_2\in\mathbb R_+$ and
$\rho\in(0, 1)$, whenever
\begin{align}
\label{eq:296}
N>D_p\delta^{-D_p}
\end{align}
for some large $D_p\in\mathbb R_+$ (possibly depending on $C_1, C_2$), it follows that for every $f\in \ell^p(\mathbb Z)$ one has
\begin{align}
\label{eq:293}
\big\|\Pi[\le \delta^{-C_1}, \le N^{-d}\delta^{-C_2}]f\big\|_{\ell^p(\mathbb Z)}
\lesssim_{p, \rho} \delta^{-C_1\rho} \|f\|_{\ell^p(\mathbb Z)}.
\end{align}
The implicit constant in \eqref{eq:293} is independent on $\delta$.
\end{theorem}

The bound in \eqref{eq:293} represents a significant quantitative
improvement compared to inequality \eqref{eq:292}. Theorem
\ref{thm:IWproj} was only recently established by the author in
collaboration with Kosz, Peluse, Wan, and Wright in \cite{KMPW}, as a
consequence of the much more general Ionescu--Wainger multifrequency
multiplier theorem for the set of canonical fractions from
\eqref{eq:4}. In fact, the authors in \cite{KMPW} provided an
affirmative answer to the Ionescu–Wainger question raised in
\cite[Remark~3, p.~361]{IW} concerning the Ionescu--Wainger
multifrequency multiplier theorem for the set of canonical
fractions. Inequality \eqref{eq:293} will be essential for
interpreting Weyl's inequality in $\ell^p(\mathbb{Z})$ spaces, see
Theorem \ref{Weyllp} below.

\smallskip

Here we gather  a few comments about Theorem \ref{thm:IWproj}.
\begin{enumerate}[label*={\arabic*}.]

\item Ionescu and Wainger in a groundbreaking paper \cite{IW},
established $\ell^p(\mathbb Z^k)$ bounds for discrete singular integral Radon
transforms. Their result follows from a deep Ionescu--Wainger multiplier theorem for the
set of the so-called \textit{Ionescu--Wainger fractions}, defined by
\begin{align*}
\tilde{\mathcal R}_{\le N} := \left\{ a/q \in\mathbb T\cap\mathbb Q: q\in P_{\le N} \text{ and } (a, q)=1 \right\},
\end{align*}
where $P_{\le N}$ is a subtle set of natural numbers with certain
prime power factorization. 

\smallskip

\item The Ionescu--Wainger multiplier theorem quickly became an
indispensable tool in the study of discrete analogues in harmonic
analysis \cite{KMT, M10, MSS, MST1, MST2, MSZ3}.  The
original Ionescu--Wainger theorem \cite{IW} implies inequality
\eqref{eq:293} for the projections from \eqref{eq:290} with the
set of the Ionescu--Wainger fractions
$\tilde{\mathcal R}_{\le \delta^{-C}}$ in place of the set of
canonical fractions $\mathcal R_{\le \delta^{-C}}$ whenever
$\log N\gtrsim \delta^{-C\gamma}$
holds for arbitrarily small $\gamma\in (0, 1)$, instead of condition \eqref{eq:296}.

\smallskip

\item The Ionescu--Wainger theorem \cite{IW} was originally developed for
scalar-valued multipliers. An important feature of their conclusion is
that the $\ell^p(\mathbb Z)$ norm of the Ionescu--Wainger multiplier
operators corresponding to the set of the Ionescu--Wainger fractions
$\tilde{\mathcal R}_{\le \delta^{-C}}$ is controlled by a multiple of
$\log(\delta^{-C}+1)^D$, where
$D := \lfloor 2\gamma^{-1}\rfloor+1$.  Their proof is based on
an intricate 
super-orthogonality phenomena. A slightly different proof with the
factor $\log(\delta^{-C}+1)$ in place of $\log(\delta^{-C}+1)^D$
follows from \cite{M10}.  The latter approach helped
clarify the role of underlying square functions and orthogonalities
(see also \cite[Section~2]{MSZ3}). The Ionescu--Wainger theory, among
other topics, was discussed by Pierce \cite{Pierce} in the context of
super-orthogonality phenomena.  Finally, we refer to the
recent paper of Tao \cite{TaoIW}, where a uniform bound in place of
$\log(\delta^{-C}+1)$ was obtained.

\end{enumerate}

\subsection{Weyl's inequality in $\ell^p(\mathbb Z)$ spaces}
The Ionescu--Wainger theorem (see Theorem \ref{thm:IWproj}) is the key
mechanism for interpreting Weyl's inequality in $\ell^p(\mathbb{Z})$
spaces.

\begin{theorem}[Weyl's inequality in $\ell^p(\mathbb Z)$ spaces]
\label{Weyllp}
Let $p\in(1, \infty)$ and $P\in\mathbb Z[{\rm n}]$ with ${\rm deg}P=d$ be fixed. For all
$C_1, C_2\in\mathbb R_+$ there exists a small constant $c\in(0, 1)$ possibly
depending on $p, C_1, C_2, P$ such that the following holds for all
$N\ge 1$ and $\delta\in(0, 1]$. For every
$f\in \ell^p(\mathbb Z)\cap\ell^2(\mathbb Z)$ whose Fourier transform vanishes on
major arcs
$\mathfrak M_{\le N^{-d}\delta^{-C_2}}(\mathcal R_{\le \delta^{-C_1}})$
we have
\begin{align}
\label{eq:1}
\|A_{N; \mathbb Z}^{P}f\|_{\ell^p(\mathbb Z)}\le c^{-1}(\delta^c+N^{-c}) \|f\|_{\ell^p(\mathbb Z)}.
\end{align}
\end{theorem}

A few remarks about the  Weyl inequality in $\ell^p(\mathbb Z)$ spaces  are in order.

\begin{enumerate}[label*={\arabic*}.]
\item Inequality \eqref{eq:1} is a far-reaching generalization of
Weyl's inequality for exponential sums \eqref{eq:288} to $\ell^p(\mathbb Z)$
spaces for all $p\in(1, \infty)$. It is important that the bounds in
\eqref{eq:1} are consistent with the bounds in the classical Weyl
inequality for exponential sums \eqref{eq:288}.

\item An essential new tool that makes these estimates available is
the Ionescu--Wainger multifrequency multiplier theorem for the set of
canonical fractions, see inequality \eqref{eq:293}.

\item The Ionescu--Wainger theory should be understood as a tool that
allows us to elevate exponential sum estimates into purely functional
analytic language. This interpretation was particularly useful in the
context of the multilinear Weyl inequality recently proved in
\cite{KMT, KMPW}.
\end{enumerate}

\begin{proof}[Proof of Theorem \ref{Weyllp}]
Suppose that
$\mathcal F_{\mathbb Z} f$ vanishes on the major arcs
${\mathfrak M}_{\leq N^{-d}\delta^{-C_2}}(\mathcal R_{\leq \delta^{-C_1}})$
for some $C_1, C_2\in\mathbb R_+$. Since
$N^{-d}\delta^{-C} \leq N^{-d}\delta^{-C_2}$ and
$\delta^{-C} \leq \delta^{-C_1}$ for $C := \frac12\min\{C_1, C_2\}$, we see that
\begin{align}
\label{eq:14}
f=(1-\Pi[\le \delta^{-C} \le N^{-d}\delta^{-C}])f, \qquad \text{ since } \qquad
\Pi[\le \delta^{-C} \le N^{-d}\delta^{-C}]f\equiv 0.
\end{align}

We fix $p\in(1, \infty)$, then we choose ${p_0} \in 2\mathbb Z$ such that
$p \in [p_0',p_0]$. By Theorem \ref{thm:IWproj}, with this choice of
${p_0} \in 2\mathbb Z$, there exists a large constant $C_{p_0} \ge 1$,
(possibly depending on $C := \frac12\min\{C_1, C_2\}$), but independent of $N$ and $\delta$, such
that if $N \ge D_{p_0}  \delta^{-D_{p_0} }$,
then for every $q \in [p_0',p_0]$, and every fixed constant $\rho \in (0, 1)$, we have
\begin{align}
\label{eq:2}
\|\Pi[\le \delta^{-C} \le N^{-d}\delta^{-C}]f\|_{\ell^{q} (\mathbb Z)}\lesssim \delta^{-\rho C}\|f\|_{\ell^{q} (\mathbb Z)},
\quad \text{ for all } \quad f\in \ell^{q} (\mathbb Z),
\end{align}
and for some $c\in(0, 1)$, by \eqref{eq:288}, we have 
\begin{align}
\label{eq:5}
|m_N(\xi)|\lesssim \delta^c,
\end{align}
whenever $\xi$ lies outside of the major arc
$\mathfrak M_{\le N^{-d}\delta^{-C}}(\mathcal R_{\le \delta^{-C}})$.

Therefore, by \eqref{eq:5} and Plancherel's theorem, we obtain
\begin{align}
\label{eq:7}
\big\|A_{N; \mathbb Z}^{P}(f-\Pi[\le \delta^{-C} \le N^{-d}\delta^{-C}]f)\big\|_{\ell^2(\mathbb Z)}
\lesssim \delta^{c}\|f\|_{\ell^2(\mathbb Z)},
\end{align}
since $(1-\Pi[\le \delta^{-C} \le N^{-d}\delta^{-C}])f$ has vanishing Fourier transform on the corresponding major arcs.
By \eqref{eq:2} we conclude that
\begin{align}
\label{eq:13}
\big\|A_{N; \mathbb Z}^{P}(f-\Pi[\le \delta^{-C} \le N^{-d}\delta^{-C}]f)\big\|_{\ell^{q}(\mathbb Z)}
\lesssim \delta^{-\rho C}\|f\|_{\ell^{q}(\mathbb Z)}.
\end{align}
Interpolating between \eqref{eq:7} and \eqref{eq:13} for $q\in\{p_0', p_0\}$, we obtain
\begin{align}
\label{eq:16}
\big\|A_{N; \mathbb Z}^{P}(f-\Pi[\le \delta^{-C} \le N^{-d}\delta^{-C}]f)\big\|_{\ell^{p}(\mathbb Z)}
\lesssim \delta^{c_p}\|f\|_{\ell^{p}(\mathbb Z)},
\end{align}
for some $c_p\in(0, 1)$, since $\rho\in(0, 1)$ is arbitrary, and can
be made as small as we need.

If $N\le D_{p_0}  \delta^{-D_{p_0} }$, then by the previous case,
the same conclusion as in \eqref{eq:16} can be derived with $N^{-c_p}$
in place of $\delta^{c_p}$. Inequality \eqref{eq:1} follows by invoking \eqref{eq:14}. This completes the proof. 
\end{proof}

\section{$r$-variational estimates for Bourgain's averages: all together}
Having all these tools from the previous section at our disposal, we are now ready to present a fairly rigorous proof of $r$-variational inequality \eqref{eq:17} for Bourgain's averages $A_{N; \mathbb Z}^P$ from \eqref{eq:24}.

We first adjust the Ionescu--Wainger projections to meet our needs. 

\subsection{Dyadic Ionescu--Wainger projections}
Let $\eta \colon \mathbb R\to[0, 1]$ be a smooth and even function
satisfying \eqref{eq:58}. For any $n, \xi\in\mathbb R$, we set $\eta_{\le n}(\xi):= \eta_{[\le 2^n]}(\xi) = \eta(2^{-n}\xi)$.
For any $l\in\mathbb N$, using definition of canonical fractions \eqref{eq:4}, we set $\Sigma_{\leq l}  :=  
\mathcal R_{\le 2^l}$, and $\Sigma_l  :=  \Sigma_{\leq l} \setminus \Sigma_{\leq l-1}.$
Then $\# \Sigma_{\leq l} \ \le \ 2^{2l}.$
Similarly, for any  $l\in\mathbb N$, and $m\in\mathbb Z$, using definition \eqref{eq:23}, we introduce the dyadic ``major arcs'' by
\begin{align*}
{\mathcal M}_{\leq l, \leq m}  :=
\mathfrak M_{\le 2^m}(\Sigma_{\le l})
\end{align*}
We note that ${\mathcal M}_{\le l, \leq m}$ is nondecreasing in both $l$ and $m$, and   if
$m \leq - 2l - 2$, then the arcs
$[\theta - 2^m, \theta + 2^m]$ that comprise
${\mathcal M}_{\leq l, \leq m}$ are pairwise disjoint. We also define
\begin{align*}
{\mathcal M}_{l, \leq m}  :=  {\mathcal M}_{\leq l, \leq m} \setminus {\mathcal M}_{\leq l-1, \leq m}
\quad \text{and} \quad
{\mathcal M}_{l,m} :=  {\mathcal M}_{l,\leq m} \setminus {\mathcal M}_{l,\leq m-1}.
\end{align*}

Using  \eqref{eq:290} we define the dyadic Ionescu--Wainger  projections $\Pi_{\leq l, \leq m} \colon \ell^2(\mathbb Z) \to \ell^2(\mathbb Z)$ by 
\begin{align}
\label{eq:20}
\Pi_{\leq l, \leq m} f(x) := \Pi [\le 2^l, \le 2^m]f(x) ,\qquad f \in \ell^2(\mathbb Z), \, x\in\mathbb Z.
\end{align}

\begin{remark}
The following properties are clear from \eqref{eq:20}.
\label{rem:2}
\begin{enumerate}[label*={(\roman*)}]
\item\label{IW4} $\Pi_{\leq l, \leq m}$ is a self-adjoint (and
real symmetric) operator. The function $\mathcal F_{\mathbb Z} (\Pi_{\leq l, \leq m} f)$ is supported on the set ${\mathcal M}_{\leq l, \leq m}$, and if
$\mathcal F_{\mathbb Z} f$ vanishes on that set, then $\Pi_{\leq l, \leq m} f \equiv 0$.
\item\label{IW5} If $m \leq - 2l - 2$ and $\mathcal F_{\mathbb Z} f$ is
supported on ${\mathcal M}_{\leq l, \leq m-2}$, then
$\Pi_{\leq l, \leq m} f = f$.
\item\label{IW6} If $m \leq - 2l - 2$, then
$\Pi_{\leq l, \leq m}$ is a contraction on $\ell^2(\mathbb Z)$.
\end{enumerate}
\end{remark}

A consequence of Theorem~\ref{thm:IWproj} is the following important bound.

\begin{proposition}
\label{mult} For every $k\in\mathbb Z_+$ and $p\in(1, \infty)$ and $\rho\in(0, 1)$ there
exists a constant $C_p\in\mathbb R_+$ such that for every $l\in\mathbb N$, and $m \in \mathbb Z$,  if
$m \leq - 6 \max\{p, p'\}(l+1)$,
then 
\begin{align}
\label{eq:116}
\| \Pi_{\leq l, \leq m} f \|_{\ell^p(\mathbb Z)} \le C_{p} (2^{C_p \rho l}\ind{l\ge 10}+\ind{l<10})\|f\|_{\ell^p(\mathbb Z)}.
\end{align}
\end{proposition}

\subsection{Minor arc estimates}
We fix a polynomial $P\in\mathbb Z[{\rm n}]$ and $1<p<\infty$. Then we
choose $p_0\in 2\mathbb Z_+$ such that $1<p_0'<p<p_0<\infty$.  The parameter
$p_0$ is a sufficiently large even integer, which serves as the
parameter $p$ from Proposition~\ref{mult} to ensure that
\eqref{eq:116} can be applied. If necessary, it may be further
adjusted to the interpolation arguments that will be used
throughout this section.

Given a large constant $C_0\in \mathbb Z_+$ possibly depending on $p_0$ and $P$, and a small constant
\begin{align*}
0<\alpha<(10^6 dp_0)^{-1},
\end{align*}
using the log-scale notation (i.e. ${\rm Log\,} N:=\lfloor \log N\rfloor$) for any $N\ge C_0$,  we define the quantities 
\begin{align*}
l_{(N)} := {\rm Log\, } N^\alpha\in\mathbb Z_+ \quad \text{and} \quad L_{(N)} := {\rm Log\,} N-l_{(N)}.
\end{align*}
Then we can write
\[
A_{N, \mathbb Z}^Pf=A_{N, \mathbb Z}^P\Pi_{\leq l_{(N)}, \leq -d L_{(N)}}f
+A_{N, \mathbb Z}^P\big(1-\Pi_{\leq l_{(N)}, \leq -d L_{(N)}}\big)f.
\]
By the Weyl inequality on $\ell^p(\mathbb Z)$ spaces  (Theorem \ref{Weyllp} with $\delta=2^{-l_{(N)}}$, $C_1=1$, and $C_2=d$) there exists a small absolute constant $c_p\in(0, 1)$ such that for all $f\in \ell^p(\mathbb Z)$ we have 
\begin{align}
\label{eq:11}
\big\|A_{N, \mathbb Z}^P\big(1-\Pi_{\leq l_{(N)}, \leq -d L_{(N)}}\big)f \big\|_{\ell^p(\mathbb Z)} 
\lesssim N^{-c_p} \|f\|_{\ell^p(\mathbb Z)}.
\end{align}

\subsection{Major arc estimates} 
\label{sec:4}
In view of \eqref{eq:11} and its summability over $N\in \mathbb D_{\tau}$, the proof of \eqref{eq:17} is reduced to showing that for all $f\in \ell^p(\mathbb Z)$ we have 
\begin{align}
\label{eq:12}
\|V^r(A_{N, \mathbb Z}^P\Pi_{\leq l_{(N)}, \leq -d L_{(N)}}f:N\in\mathbb D_{\tau}) \|_{\ell^p(\mathbb Z)} 
\lesssim \|f\|_{\ell^p(\mathbb Z)}.
\end{align}

Now we need to fix a convenient notation. For any bounded function $\mathfrak m : \mathbb T\to\mathbb C$,
any finite set of frequencies $\Sigma\subseteq \mathbb T$, any bounded function $S:\Sigma\to \mathbb C$, we define the Fourier multiplier operator  by 
\begin{align*}
T_{\mathbb Z}^{\Sigma}[S; \mathfrak m]f(x) := \mathcal F_{\mathbb Z}^{-1}\Big[\sum_{\theta\in\Sigma}S(\theta)\tau_{\theta}\mathfrak m\mathcal F_{\mathbb Z}f\Big](x), \qquad  x\in\mathbb Z,
\end{align*}
for any test function $f \colon \mathbb Z\to\mathbb C$,
where $\tau_{\theta}\mathfrak m(\xi) := \mathfrak m(\xi-\theta)$ for $\xi\in\mathbb T$.

Writing
\begin{align*}
A_{N, \mathbb Z}^P\Pi_{\leq l_{(N)}, \leq -d L_{(N)}}f=\sum_{0\le l\le l_{(N)}}
A_{N, \mathbb Z}^P(\Pi_{l, \leq -dL_{(N)}}f),
\end{align*}
where $\Pi_{l, \leq -d L_{(N)}} := T_{\mathbb Z}^{\Sigma_{l}}[1; \eta_{\le -d L_{(N)}}]$, it suffices to show for all $l\in\mathbb N$ that
\begin{align}
\label{eq:27}
\|V^r(A_{N, \mathbb Z}^P\Pi_{l, \leq -d L_{(N)}}f:N\in\mathbb D_{\tau, l}) \|_{\ell^p(\mathbb Z)} 
\lesssim 2^{-c_pl}\|f\|_{\ell^p(\mathbb Z)},
\end{align}
where
$\mathbb D_{\tau, l}:=\{N\in\mathbb D_{\tau}:N\ge C_0 \text{ and
} N\ge 2^{l/\alpha}\}$. Summing \eqref{eq:27} over $l\in\mathbb N$,
we obtain \eqref{eq:12}.

We now take the advantage of the fact that each function in
\eqref{eq:27} is restricted to major arcs.
For $N\ge 1$ we define a continuous counterpart of $m_N$ by
\begin{align}
\label{eq:6}
\mathfrak{m}_N(\xi) := \int_{0}^1e(\xi\cdot P(Nt))dt, \qquad \xi\in\mathbb R.
\end{align}
For every $a\in\mathbb Z^k$ and $q\in\mathbb Z_+$ such that $(a, q)=1$ we  define the complete exponential sum by
\begin{align}
\label{eq:10}
G(a/q) := m_q(a/q)=\mathbb E_{n \in [q]} e\Big(\frac{a}{q} P(n)\Big).
\end{align}

\begin{lemma}
\label{lemma:1}
Let $P\in\mathbb Z[{\rm n}]$ with $\deg P=d$ be given. Then there is $C_{P} \in \mathbb R_+$ such
that for every $N \geq 1$, $M \in \mathbb R_+$, and
$l\in\mathbb N$ the following holds.  For every $\xi\in\mathbb T$
and $\theta\in \Sigma_{l}$ such that
$|\xi-\theta|\le M^{-1}$, one has
\begin{align}
\label{eq:32}
|m_N(\xi)-G(\theta)\mathfrak m_N(\xi-\theta)|\le C_{P}2^{l}\big(M^{-1}N^{d-1}+N^{-1}\big),
\end{align}
with $m_N, \mathfrak m_N, G$ defined in \eqref{eq:28}, \eqref{eq:6} and \eqref{eq:10} respectively.
\end{lemma}

\begin{proof}
Inequality \eqref{eq:32} follows by a simple application of the mean-value theorem.
\end{proof}

Using Lemma \ref{lemma:1} with $M=2^{dL_{(N)}}$ we conclude that
\[
A_{N, \mathbb Z}^P(\Pi_{l, \leq -dL_{(N)}}f)=T_{\mathbb Z}^{\Sigma_{l}}[G; \mathfrak m_N\eta_{\le -d L_{(N)}}]f+E_Nf.
\]
Moreover by a simple interpolation there exists $c_p\in(0, 1)$ such that
\begin{align*}
\|E_Nf\|_{\ell^p(\mathbb Z)} 
\lesssim N^{-c_p}\|f\|_{\ell^p(\mathbb Z)} \quad \text{ for all } \quad f\in \ell^p(\mathbb Z).
\end{align*}
This inequality for $N\in\mathbb D_{\tau, l}$ reduces inequality \eqref{eq:27} to proving that for all $l\in\mathbb N$, we have 
\begin{align}
\label{eq:40}
\|V^r(T_{\mathbb Z}^{\Sigma_{l}}[G; \mathfrak m_N\eta_{\le -d L_{(N)}}]f:N\in\mathbb D_{\tau, l}) \|_{\ell^p(\mathbb Z)}
\lesssim 2^{-c_pl}\|f\|_{\ell^p(\mathbb Z)}.
\end{align}

For this purpose, define $u:=100(l+1)$ and note that for any $N\in\mathbb D_{\tau, l}$ we have 
$N\ge \max\{2^{l/\alpha}, C_0\},$
which  implies $N\ge 2^{10^5  dp_0l}C_0^{1/2}\ge 2^{100p_0du},$
provided that $C_0\ge 2^{10^6  dp_0}$. Thus we may conclude that
\begin{align}
\label{eq:30}
T_{\mathbb Z}^{\Sigma_{l}}[G; \mathfrak m_N\eta_{\le -d L_{(N)}}]f=T_{\mathbb Z}^{\Sigma_{l}}[1; \mathfrak m_N\eta_{\le -d L_{(N)}}]T_{\mathbb Z}^{\Sigma_{l}}[G; \eta_{\le -d up_0}]f.
\end{align}
By the semi-norm Ionescu--Wainger theorem (see \cite[Theorem 3.32]{KMPW}), we have for all $l\in\mathbb N$, that
\begin{align}
\label{eq:3}
\|V^r(T_{\mathbb Z}^{\Sigma_{l}}[1; \mathfrak m_N\eta_{\le -d L_{(N)}}]f:N\in\mathbb D_{\tau, l}) \|_{\ell^p(\mathbb Z)} 
\lesssim_{\rho} 2^{\rho l}\|f\|_{\ell^p(\mathbb Z)},
\end{align}
for any $\rho\in(0, 1)$.
 On the other hand, using Lemma \ref{lemma:1} again,  with $M=2^{dup_0}$ we conclude that 
\[
T_{\mathbb Z}^{\Sigma_{l}}[G; \eta_{\le -dup_0}]f=A_{2^u, \mathbb Z}^P(\Pi_{l, \leq -dup_0}f)+E_{2^u}f,
\]
where the error term $E_{2^u}$ satisfies for some $c_p\in(0, 1)$ the following bound
\begin{align*}
\|E_{2^u}f\|_{\ell^p(\mathbb Z)} 
\lesssim 2^{-c_pl}\|f\|_{\ell^p(\mathbb Z)} \quad \text{ for all } \quad f\in \ell^p(\mathbb Z).
\end{align*}
By the Weyl inequality on $\ell^p(\mathbb Z)$ spaces, we conclude that
\begin{align*}
\|A_{2^u, \mathbb Z}^P(\Pi_{l, \leq -dup_0}f)\|_{\ell^p(\mathbb Z)} 
\lesssim 2^{-c_pl}\|f\|_{\ell^p(\mathbb Z)} \quad \text{ for all } \quad f\in \ell^p(\mathbb Z),
\end{align*}
since $\Pi_{l, \leq -dup_0}f$ has vanishing Fourier transform on $\mathcal M_{\le l-1, \leq -du+\varepsilon(l-1)}$ for a sufficiently small $\varepsilon>0$. Combining the last two bounds with \eqref{eq:30} and inequality \eqref{eq:3}, since $\rho\in(0, 1)$ can be made arbitrarily small, we obtain \eqref{eq:40} as desired. This completes the proof of inequality \eqref{eq:17}.

\section{Furstenberg--Bergelson--Leibman conjecture}

One of the major and best known open problems in pointwise ergodic
theory is the following far reaching conjecture of
Furstenberg--Bergelson--Leibman:

\begin{ber}
Let $d, k, m\in\mathbb N$ be given and suppose that a family of invertible
measure-preserving transformations $T_1,\ldots, T_d:X\to X$ of a
probability measure space $(X, \mathcal B(X), \mu)$ generates a
nilpotent group. Assume that
$P_{1, 1},\ldots,P_{i, j},\ldots, P_{d, m}\in \mathbb Z[\mathrm n_1,\ldots, \mathrm n_k]$
are $k$-variate polynomials with integer coefficients. Then for any functions
$f_1, \ldots, f_m\in L^{\infty}(X)$, the non-conventional multilinear
polynomial averages
\begin{align}
\label{eq:29}
A_{N_1,\ldots, N_k; X,  T_1,\ldots, T_d}^{P_{1, 1}, \ldots, P_{d, m}}(f_1,\ldots, f_m)(x)=\mathbb E_{n\in[N_1]\times\ldots\times [N_k]}\prod_{j=1}^mf_j(T_1^{P_{1, j}(n)}\cdots T_d^{P_{d, j}(n)} x)
\end{align}
converge for $\mu$-almost every $x\in X$ as
$\min\{N_1,\ldots, N_k\}\to\infty$.  If $N_1=\ldots=N_k=N$ we shall
abbreviate
$A_{N_1,\ldots, N_k;X, T_1,\ldots, T_d}^{P_{1, 1}, \ldots, P_{d, m}}$
to $A_{N; X, T_1,\ldots, T_d}^{P_{1, 1}, \ldots, P_{d, m}}$.

\end{ber}

This conjecture represents a significant challenge in pointwise
ergodic theory and modern Fourier analysis. It was initially proposed
in person by Furstenberg (see Austin's article \cite[p. 6662]{A1})
before being published by Bergelson and Leibman \cite[Section 5.5,
p. 468]{BL} (see also \cite{Ber1, Ber2, Fra}). Bergelson and Leibman
\cite{BL} also demonstrated that convergence in \eqref{eq:29} may fail
if the transformations \( T_1, \ldots, T_d \) generate a solvable
(non-nilpotent) group, suggesting that the nilpotent setting is likely
the appropriate context for this conjecture.

To understand how the Furstenberg--Bergelson--Leibman conjecture
arose, one must go back to 1936, when Erd{\"o}s and Tur{\'a}n
formulated an audacious conjecture asserting that every
\( A \subseteq \mathbb{Z} \) with positive upper density must contain
an arithmetic progression of any length \( m \in \mathbb{Z}_+ \),
which is a configuration of the form
\[
x, x+n, x+2n, \ldots, x+(m-1)n \in A,
\]
for some $x\in A$ and $ n\in\mathbb Z\setminus\{0\}$.

In 1953, the conjecture of Erd{\"o}s and Tur{\'a}n was proved by Roth
\cite{Rot} for \( m = 3 \) using Fourier analysis. In 1974,
Szemer{\'e}di \cite{Sem1} verified this conjecture for arbitrary
\( m \in \mathbb{Z}_+ \), employing intricate combinatorial arguments
based on a famous regularity lemma, which is a powerful tool in graph theory.

At first glance, ergodic theory and extremal combinatorics may seem
entirely unrelated. However, in the late 1970s, Furstenberg
\cite{Fur0} had the remarkable insight to recognize that
Szemer{\'e}di's theorem can be interpreted as a multiple recurrence
result in ergodic theory, akin to the classical Poincar{\'e}
recurrence theorem. He then formulated what is now known in the
literature as the Furstenberg correspondence principle, which serves
as a bridge between combinatorics and ergodic theory, providing a
conceptually new proof of Szemer{\'e}di's theorem.

Furstenberg's proof \cite{Fur0} of Szemer{\'e}di's theorem \cite{Sem1}
was the starting point of a new field called ergodic Ramsey theory,
which has led to many natural generalizations of Szemer{\'e}di's theorem,
including a polynomial Szemerédi theorem by Bergelson and Leibman
\cite{BL1}. The latter result (in its simplest form) asserts that for
any \( m \in \mathbb{Z}_+ \) and for every
\( A \subseteq \mathbb{Z} \) with positive upper density, and for
arbitrary polynomials \( P_1, \ldots, P_m \in \mathbb{Z}[n] \) with
vanishing constant terms, one has
\[
x, x + P_1(n), x + P_2(n), \ldots, x + P_m(n) \in A,
\]
for some \( x \in A \) and \( n \in \mathbb{Z} \setminus \{0\} \). In
other words, \( A \) contains polynomial patterns (or polynomial progressions) of
arbitrary length.

Both Furstenberg's theorem \cite{Fur0} and the
theorem by Bergelson and Leibman \cite{BL1} served as catalysts for
understanding the limiting behavior of multiple averages
\eqref{eq:29}, which are natural tools for detecting recurrent points
and, consequently, arithmetic and polynomial progressions in subsets
of integers with non-vanishing upper density. This provides the
motivation for the Furstenberg--Bergelson--Leibman conjecture.

\subsection{Norm convergence for multilinear averages}
The Furstenberg--Bergelson--Leibman conjecture was initially
focused on the convergence of \eqref{eq:29} in both the
\( L^2(X) \) norm and pointwise almost everywhere. The \( L^2(X) \)
norm convergence for
$A_{N; X, T_1,\ldots, T_d}^{P_{1, 1}, \ldots, P_{d, m}}$ was settled
by Walsh in a breakthrough work \cite{W}. We also refer to \cite{A2,
PZK} for alternative proofs. Although Walsh's result establishes
\( L^2(X) \) norm convergence for
$A_{N; X, T_1,\ldots, T_d}^{P_{1, 1}, \ldots, P_{d, m}}$ (even for
noncommutative transformations \( T_1, \ldots, T_d \) spanning a
nilpotent group), the question of identifying the limit for arbitrary
transformations and polynomials remains largely
unanswered. Identifying the limit for general polynomial ergodic
averages is a well-known open problem in ergodic theory, even in the
linear case.

Before Walsh's paper \cite{W}, there was an extensive body of research
aimed at establishing \( L^2(X) \) norm convergence for \eqref{eq:29}
in the single transformation case, i.e., with \( d=k=1 \),
\(m \in \mathbb{Z}_+ \), \( P_{1,j} = P_j \in \mathbb{Z}[n] \) for
\( j \in [m] \), where \( P_1, \ldots, P_m \) are arbitrary
polynomials. This includes fundamental work for linear polynomials by
Host and Kra \cite{HK} and, independently, by Ziegler \cite{Z1}.  The
case of general polynomials was addressed in the work by Leibman
\cite{Leibman}. In the single transformation case, in
contrast to Walsh's paper \cite{W}, the limiting function in
\eqref{eq:29} can be identified thanks to the theory of Host–Kra
factors \cite{HK} and equidistribution on nilmanifolds \cite{Leibman}.
We refer to \cite{Fra}, where additional information on this topic and
precise references can be found.

The case of arbitrary commuting measure-preserving transformations is
much more delicate, i.e. \eqref{eq:29} with \( k=1 \),
\( d=m \in \mathbb{Z}_+ \), \( P_{j,j} = P_j \in \mathbb{Z}[n] \) for
\( j \in [m] \), and \( P_{i,j} = 0 \) for \( i \neq j \), where
\( P_1, \ldots, P_m \) are arbitrary polynomials and
\( T_1, \dots, T_d \) are commuting transformations.  The case of
linear polynomials was subsequently studied by Tao~\cite{Tao}, Austin
\cite{A2}, and Host \cite{H}. In \cite{CFH}, Chu, Frantzikinakis, and
Host established $L^2(X)$ norm convergence for these averages when the
polynomials have distinct degrees.  However, the $L^2(X)$ limit for
$A_{N; X, T_1,\ldots, T_d}^{P_{1, 1}, \ldots, P_{d, m}}$ in the case
of commuting transformations and linearly independent polynomials has
recently been identified by Frantzikinakis and Kuca \cite{FK1}.

The state of knowledge is dramatically worse for pointwise almost
everywhere convergence of multilinear polynomial ergodic averages. As
we noted at the beginning of this survey, the phenomenon of pointwise
convergence, while the most natural mode of convergence, is very
subtle and can differ significantly from norm convergence, making the
study of pointwise convergence problems quite challenging due to the
requirement of a variety of quantitative tools from different areas.

\subsection{Furstenberg--Bergelson--Leibman conjecture: commutative linear case}
In \cite{M10}, the author presented a slightly different proof of the
Ionescu--Wainger multiplier theorem \cite{IW}. More importantly, in
\cite{M10}, it was shown for the first time that the Ionescu--Wainger
multiplier theorem can be interpreted as a discrete variant (a
multifrequency variant) of the Littlewood--Paley theory, which
captures all the Diophantine features arising from the
Hardy--Littlewood circle method. This answered in the affirmative a
question posed by Stein in the early 1990s, which had been an open
problem in the field of discrete analogues in harmonic analysis since
Bourgain's papers \cite{B1, B2, B3}.

The discrete Littlewood--Paley theory from \cite{M10} resulted in
$r$-variational ergodic theorems \cite{MST1, MST2}, jump ergodic
theorems \cite{MSZ1, MSZ2, MSZ3}, and, recently, oscillation ergodic
theorems \cite{MSS}, particularly establishing the
Furstenberg--Bergelson--Leibman conjecture for
\begin{align*}
A_{N; X,  T_1,\ldots, T_d}^{P_{1, 1}, \ldots, P_{d, 1}}f(x)
=\mathbb E_{n\in[N]^k}f(T_1^{P_{1, 1}(n)}\cdots T_d^{P_{d, 1}(n)} x),
\end{align*}
corresponding to \eqref{eq:29} with $m=1$, $d, k\in\mathbb N$,
$N_1=\ldots=N_k=N$ for any polynomials
$P_{1, 1},\ldots,P_{d, 1}\in \mathbb Z[\mathrm n]$, and for any invertible
measure-preserving commuting transformations $T_1,\ldots, T_d:X\to X$
on a $\sigma$-finite measure space $(X, \mathcal B(X), \mu)$.  The
results from \cite{MST1, MST2, MSS, MSZ1, MSZ2, MSZ3, MT} also give
the sharpest possible quantitative information about the pointwise
convergence phenomena, and secured the Rademacher--Menshov inequality
\cite{BMSW1, BMSW1',  MT} as one of the primary tools in the
subject.

\subsection{Furstenberg--Bergelson--Leibman conjecture: multilinear case}
Thirty years after  foundational bilinear Bourgain's paper
\cite{B0} on the double recurrence theorem (i.e., for $A_{N; X, T}^{a\mathrm n, b\mathrm n}$ that
corresponds to \eqref{eq:29} with $d=k=1$, $m=2$, $P_{1,1}(n)=an$, and
$P_{1, 2}(n)=bn$ for $a, b\in\mathbb Z$), the author with Krause and Tao
\cite{KMT} proved the  Furstenberg--Bergelson--Leibman conjecture
for the polynomial Furstenberg--Weiss averages
\[
A_{N; X, T}^{\mathrm n, P(\mathrm n)}(f_1, f_2)(x)=\mathbb E_{n\in[N]} f_1(T^{n}x) f_2(T^{P(n)}x),
\]
corresponding to \eqref{eq:29} with $d=k=1$, $m=2$, $P_{1,1}(n)=n$,
and $P_{1, 2}(n)=P(n)$, where $P\in\mathbb Z[\mathrm n]$ with
${\rm deg }P\ge2$, see \cite{FurWei}. This work has been
greatly generalized recently by \cite{KMPW}, where pointwise almost
everywhere convergence was established for genuinely multilinear
ergodic averages of arbitrary polynomials with distinct degrees,
breaking the bilinearity barrier in the Furstenberg--Bergelson--Leibman
conjecture for the first time.

Aside from resolving the Furstenberg--Bergelson--Leibman conjecture
for multilinear ergodic averages for arbitrary polynomials with
distinct degrees, the pointwise convergence phenomenon in \cite{KMPW, KMT}
has been quantified by establishing sharp $r$-variational estimates.
The proofs from \cite{KMT} and \cite{KMPW} integrate novel ideas
from number theory, additive combinatorics, and Fourier analysis,
providing fresh insights into numerous unresolved problems in these
fields.  The difficulties encountered in the multilinear setting
necessitated significant new ideas.  Most notably, Plancherel's
theorem and Weyl sum estimates (see inequality \eqref{eq:288} and
\eqref{eq:55}) are insufficient on their own to control the
contribution from the minor arcs, thus defeating a classical
implementation of the circle method introduced by Bourgain \cite{B3}.
On the minor arcs, we used the inverse theory of Peluse and
Prendiville \cite{PP1}, \cite{P2} combined with the Hahn--Banach
theorem, the Ionescu--Wainger multiplier theorem \cite{IW}, and
$L^p$-improving estimates to construct a robust \textit{multilinear circle
method}. This resulted in a multilinear counterpart of Weyl's
inequality from Theorem \ref{Weyllp}. On the major arcs, we developed
certain paraproduct estimates based on discrete Littlewood--Paley
theory \cite{M10}, combined with Sobolev smoothing estimates and
$L^p$-improving estimates to control high and low frequencies.

\subsection{Furstenberg--Bergelson--Leibman conjecture: non-commutative linear case}
Recently, the author with Ionescu, Magyar, and Szarek \cite{IMMS}
established the first general nilpotent case of the
Furstenberg--Bergelson--Leibman conjecture for linear averages
\begin{align*}
A_{N; X,  T_1,\ldots, T_d}^{P_{1, 1}, \ldots, P_{d, 1}}f(x)
=\mathbb E_{n\in[N]}
f(T_1^{P_{1, 1}(n)}\cdots T_d^{P_{d, 1}(n)} x),
\end{align*}
corresponding to \eqref{eq:29} with $k=m=1$, $d\in\mathbb N$ and any
polynomials $P_{1, 1},\ldots,P_{d, 1}\in \mathbb Z[\mathrm n]$, and for any
invertible measure preserving transformations $T_1,\ldots, T_d:X\to X$
on a $\sigma$-finite measure space $(X, \mathcal B(X), \mu)$ that
generate a nilpotent group of step two. One can think that the
following commutator relations $[[T_i, T_j], T_l]={\rm Id}$ are
satisfied for all $i, j, l\in[d]$.  

Prior to \cite{IMMS}, only partial results \cite{IMSW, IMW, MSW1} were
known in the discrete nilpotent setting. Thanks to the Calder{\'o}n
transference principle \cite{Cald}, noncommutative ergodic operators
from \cite{IMMS}, can be interpreted as convolution operators on
certain universal nilpotent groups $(\mathbb G_0, \cdot)$ of step
two. These groups $(\mathbb G_0, \cdot)$ are induced by the underlying
measure-preserving transformations $T_1,\ldots, T_d$ and the
polynomial mappings $P_{1, 1},\ldots, P_{d, 1}$.  However, \cite{IMMS}
represents another example where Fourier methods are not directly
available. The main issue is that there is no suitable Fourier
transform on nilpotent groups --- at the analytical precision level of the
classical circle method --- that is compatible with the structure of the
underlying convolution operators. These challenges lead to significant
difficulties in the proof and necessitate substantial new ideas that were developed in \cite{IMMS}:
\begin{itemize}
\item[(i)] Classical Fourier techniques have been replaced with
almost-orthogonality methods that exploit high-order \( TT^* \)
arguments for operators defined on the discrete group
\( \mathbb{G}_0 \), which arise in the proof of the main theorem in
\cite{IMMS}. Studying high powers of \( TT^* \) (i.e., \( (TT^*)^r \)
for a large \( r \in \mathbb{Z}_+ \)) allows for a simple heuristic
underlying the proof of Waring-type problems to be adapted 
 in the context of our proof. This heuristic states
that the more variables that occur in the Waring-type equation, the
easier it is to find a solution. By manipulating the parameter \( r \)
(usually taking \( r \) to be very large), we can always determine how
many variables we have at our disposal, making the operators in our
questions ``smoother and smoother''.

\item[(ii)]
Our main new construction in \cite{IMMS} is what we call a \textit{nilpotent circle method}, an iterative procedure that starts from the center of the group and moves down along its central series. This method allows us to use some ideas from the classical circle method recursively at every stage. In the case of nilpotent groups of step two, the procedure consists of two basic iterations and one additional step corresponding to ``major arcs''.  The analysis of minor arcs requires two types of Weyl's inequalities: the classical one (see inequality \eqref{eq:288}) and the nilpotent version inspired by Davenport \cite{Dav} and Birch \cite{Bi}, which was proved in \cite{IMW}. The analysis of major arcs brings into play tools that combine continuous harmonic analysis on groups with arithmetic harmonic analysis over finite integer rings.
\end{itemize}

\subsection{Furstenberg--Bergelson--Leibman conjecture: commutative multiparameter linear case}

A multiparameter ergodic theorem, independently proved by Dunford
\cite{D} and Zygmund \cite{Z}, establishes the
Furstenberg--Bergelson--Leibman conjecture for $m=1$, $d=k\in\mathbb N$ with
linear polynomials $P_{1, 1}(n)=\ldots=P_{d, 1}(n)=n$. In 2016,
Bourgain, motivated by the results \cite{D, Z} and a multiparameter
equidistribution theorem of Arkhipov, Chubarikov, and Karatsuba
\cite{ACK}, proposed studying the pointwise convergence problem for
the multiparameter averages
\[
A_{N_1,\ldots, N_k; X,  T}^{P}f(x)=\mathbb E_{n\in[N_1]\times\ldots\times [N_k]}f(T^{P(n)}x),
\]
which was established by the author with Bourgain, Stein, and Wright
in \cite{BMSWer} giving an affirmative answer to a multiparameter
variant of a question of Bellow and Furstenberg from the late
1980's. This also contributes to the Furstenberg--Bergelson--Leibman
conjecture in the case $d=m=1$, $k\in\mathbb N$, and arbitrary
$P_{1,1}\in \mathbb Z[\mathrm n_1,\ldots, \mathrm n_k]$. A continuous
counterpart of this problem, in a more general setting, was recently
completely resolved by the author in collaboration with Kosz,
Langowski, and Plewa in \cite{KLMP}.  In \cite{BMSWer}, a
\textit{multiparameter circle method} was developed to address the
challenges arising from multiparameter phenomena in ergodic theory and
discrete harmonic analysis. 

This topic has historically been both a
burden and a challenge in classical harmonic analysis, beginning with
Fefferman's counterexample for double Fourier series, continuing with
the failure of Lebesgue's differentiation theorem in \(L^1\) for
averages over rectangles with sides parallel to the coordinate axes,
and culminating with the Kakeya maximal conjecture, which remains a
major open problem in modern harmonic analysis \cite{bigs}.

\end{document}